\documentclass{amsart}
\usepackage{amssymb}
\usepackage{latexsym}
\usepackage{xypic}
\usepackage{pstricks-add}

\newtheorem{theorem}{Theorem}[section]

\theoremstyle{definition}
\newtheorem{definition}[theorem]{Definition}

\theoremstyle{remark}

\newcommand{\SL}{\textrm{SL}}

\newcommand{\Flecha}[1]{\overleftrightarrow{#1}}
\newcommand{\C}{\mathbb{C}}
\newcommand{\Z}{\mathbb{Z}}

\newcommand{\Proy}{\mathbb P_{\mathbb C}}
\newcommand{\Hip}{\mathbb{H}}
\def\P{{\mathbb P}_\mathbb C}
\def\PSL{{\rm PSL}}
\def\SL{{\rm SL}}

\def\Eq{{\rm Eq}}
\def\Kul{{\rm Kul}}
\def\Myr{{\rm Myr}}

\def\SP{{\rm SP}}

\begin{document}
	\title{  \sc{ Elementary groups in $\PSL(3,\C)$ }}
	\author{W. Barrera}
	\address{ Facultad de Matem\'aticas, Universidad
		Aut\'onoma de Yucat\'an, Anillo Perif\'erico Norte Tablaje Cat
		13615, M\'erida, Yucat\'an, M\'exico. } \email{bvargas@uady.mx}
	
	\author{A. Cano}
	\address{Instituto de Matem\'aticas UNAM , Av. Universidad S/N,
		Col. Lomas de Chamilpa, C.P. 62210, Cuernavaca, Morelos, M\'exico.}
	\email{angelcano@im.unam.mx}
	
	\author{J.P. Navarrete}
	\address{ Facultad de Matem\'aticas, Universidad
		Aut\'onoma de Yucat\'an, Anillo Perif\'erico Norte Tablaje Cat
		13615, M\'erida, Yucat\'an, M\'exico. } \email{jp.navarrete@uady.mx}
	
	\author{J. Seade}
	\address{Instituto de Matem\'aticas UNAM , Av. Universidad S/N,
		Col. Lomas de Chamilpa, C.P. 62210, Cuernavaca, Morelos, M\'exico.}
	\email{jseade@im.unam.mx}
	
	\date{}
	

	
	\begin{abstract}
		In this paper, we give a classification of the subgroups
		of $\textrm{PSL}(3, \mathbb{C})$  that act on $\Proy ^2$ in such a way that their   Kulkarni
		limit set has finitely many lines in general position lines. These are the  elementary groups.
		\\ $\,$  \\ $\,$ \\ \centerline{$ \qquad;\qquad\qquad \qquad\qquad\qquad \qquad\qquad$ To Ravi Kulkarni, with great admiration} 
		
	\end{abstract}
	
	\maketitle
	\section*{Introduction}
	Kleinian groups are discrete subgroups of $\textrm{PSL}(2,
	\mathbb{C})$ acting on $\mathbb{S}^2 \cong \Proy ^1 $ in such a way that their
	limit set is not all of $\mathbb{S}^2$. These are classified into
	elementary and non-elementary groups. The elementary groups are
	those   whose limit set consists of $0,1$ or $2$
	points, and they are classified, see for instance
	(see \cite{Ma}).

	In this work we look at complex Kleinian groups, that is,
	discrete subgroups of $\textrm{PSL}(3, \mathbb{C})$ acting properly
	and discontinuously on some open subset of $\Proy ^2$.
	In
	\cite{CS}
	it is proved that the Kulkarni limit set $\Lambda_{\rm Kul}$ (see Definition
	\ref{kuldef})
	of every infinite complex
	Kleinian group contains at least one complex projective line. In fact we known from
	\cite{BCN1}
	that  under some
	generic hypothesis on the group, the Kulkarni limit set is a union of
	complex projective lines in general position, {\it i.e.} no three of them intersect.
	By definition, the exception are the elementary groups. These  have finitely many lines in general position in their limit set, and there are two types ef such groups:  those of the first kind have a finite number of lines in their limit set, and
	the groups of the second kind have infinitely many lines in their limit set, but only finitely many in general position.
	
	As an example of elementary groups of the first kind, take the cyclic group generated by a $3 \times 3$ diagonal matrix with non-zero eigenvalues of different norms; this has two projective lines as limit set in $\P^2$. On the other hand (see
	\cite{CPS}), 
	take a $\C$-Fuchsian group in $\PSL(3,\C)$. Then its limit set in $\P^2$ consists of a cone of projective lines with base a  circle (see
	\cite{CPS}
	;  it has infinitely many lines, but all of them pass through the vertex of the cone, so there are only two in general position. We remark that if we take an $\mathbb R$-Fuchsian group in 
	$\PSL(3,\C)$, then it is non-elementary and its limit set actually has
	infinitely many lines in general position, see
	\cite{CPS}

	In this work we study and describe all the elementary groups in  $\textrm{PSL}(3, \mathbb{C})$.
	
	An interesting class of elementary groups are the discrete purely parabolic groups. One finds that the limit set of such a group consists of either one line, or a cone of lines over the circle, or the whole of $\P^2$. We describe these in Section \ref{sec parabolic}.

	We know (see Theorem \ref{th. number of lines} below) that if $\Gamma \subset \PSL(3,\mathbb{C})$ is an infinite discrete subgroup, then
	the number of complex projective lines in  its limit set $\Lambda_{\Kul}(\Gamma)$
	is equal to $1, 2, 3$ or $\infty$, and
	the number of lines $\Lambda_{\Kul}(\Gamma)$ in general position can be $1, 2, 3, 4$ or $\infty$. Also,  
	if the number of lines in $\Lambda_{\Kul}(\Gamma)$ is exactly 3, then these lines are in general position.    
	
	It follows from Theorem \ref{th. number of lines} and the results in this note, that if $\Lambda_{\Kul}(\Gamma)$ has more than one line, then it is a union of lines. The only case where $\Lambda_{\Kul}(\Gamma)$ can have  isolated points is when this set actually consists of one line and one point.

	Section \ref{sec 1} contains the basic definitions and preliminary results that we use in the sequel. Sections \ref{s:oneline} and \ref{s:fourlines} are expository and we cite in each the appropriate bibliography. The results  in these sections are taken from
	\cite{BCN3} and \cite{BCN2}
	and we include them here for completeness. In
	\cite{BCN3}
	the authors discuss the groups with limit set a line, and
	in
	\cite{BCN2}
	those with four lines in general position.

	Section \ref{sec 4} focuses on the case of  two lines. These can be of the first kind, with exactly 2 lines in $\Lambda_{\rm Kul}$, or of the second kind, with limit set a cone of lines over a circle, but only 2 in general position.

	In section \ref{sec 5} we describe the groups with exactly   three lines in their limit set, and
	section \ref{s:isolated} discusses the only case where the limit set is not a union of lines. Then   $\Lambda_{\rm Kul}$  consists of a line and a point.
	
	\section{Preliminaries}\label{sec 1}

	\subsection{Pseudo-projective transformations}
	We let $\mathbb{P}^2_\mathbb{C}:=(\mathbb{C}^{3}\setminus \{0\})/\mathbb{C}^*,$ be
	the complex projective plane.
	This  is   a  compact connected  complex $2$-dimensional manifold.
	Let $[\mbox{ }]:\mathbb{C}^{3}\setminus\{0\}\rightarrow
	\mathbb{P}^{2}_{\mathbb{C}}$ be   the quotient map. If
	$\beta=\{e_1,e_2,e_3\}$ is the standard basis of $\mathbb{C}^3$, we
	write $[e_j]=e_j$ and if $w=(w_1,w_2,w_3)\in
	\mathbb{C}^3\setminus\{0\}$ then we  write $[w]=[w_1:w_2:w_3]$.
	Also, $\ell\subset \mathbb{P}^2_{\mathbb{C}}$ is said to be a
	complex line if $[\ell]^{-1}\cup \{0\}$ is a complex linear subspace
	of dimension $2$. Given  $p,q\in \mathbb{P}^2_{\mathbb{C}}$ distinct
	points,     there is a unique complex projective line passing
	through $p$ and $q$, such complex projective line is called a
	\emph{line}, for short, and it is denoted by
	$\overleftrightarrow{p,q}$.
	
	Consider the action of $\mathbb{Z}_{3}$ (viewed as the cubic roots
	of the unity) on  $\textrm{SL}(3,\mathbb{C})$ given by the usual
	scalar multiplication, then the Lie group
	$$\textrm{PSL}(3,\mathbb{C})=\textrm{SL}(3,\mathbb{C})/\mathbb{Z}_{3}$$ is
	the group of automorphisms of $\P^2$.
	Let $[[\mbox{
	}]]:\textrm{SL}(3,\mathbb{C})\rightarrow \textrm{PSL}(3,\mathbb{C})$
	be   the quotient map,   $\gamma\in \textrm{PSL}(3,\mathbb{C})$ and
	$\widetilde\gamma\in \textrm{GL}(3,\mathbb{C})$, we say that
	$\tilde\gamma$ is a  lift of $\gamma$ if there is a cubic root
	$\tau$ of $Det(\gamma)$ such that $[[\tau
	\widetilde\gamma]]=\gamma$. Also, we use the notation
	$(\gamma_{ij})$ to denote elements  in $\textrm{SL}(3,\mathbb{C})$.

	We denote by $M_{3 \times 3}(\mathbb{C})$ the space of all $3 \times 3$
	matrices with entries in $\mathbb{C}$ equipped with the standard
	topology. The quotient space
	\[ \SP(3,\C):= (M_{3 \times 3}(\mathbb{C}) \setminus \{\mathbf{0}\})/\mathbb{C}^* \]
	is called the space of \emph{pseudo-projective maps of} $\mathbb{P}
	_{\mathbb{C}} ^2$ and it is naturally identified with the projective
	space $\mathbb{P} _{\mathbb{C}} ^8$. Since $\textrm{GL}(3,\mathbb{C})$ is an
	open, dense, $\mathbb{C} ^*$-invariant set of $M_{3 \times 3}(\mathbb{C})
	\setminus \{\mathbf{0}\}$, we obtain that the space of
	pseudo-projective maps of $\mathbb{P} _{\mathbb{C}} ^2$ is a
	compactification of $\textrm{PSL}(3, \mathbb{C})$. As in the case of
	projective maps, if $\mathbf{s} \in M_{3 \times 3}(\mathbb{C})
	\setminus \{\mathbf{0}\}$, then $[\mathbf{s}]$ denotes the
	equivalence class of the matrix $\mathbf{s}$ in the space of
	pseudo-projective maps of $\mathbb{P} _{\mathbb{C}} ^2$. Also, we say that
	$\mathbf{s}\in M_{3 \times 3}(\mathbb{C})\setminus \{\mathbf{0}\}$ is a
	lift of the pseudoprojective map $S$, whenever $[\mathbf{s}]=S$.
	
	Let $S$  be an element in $( M_{3 \times 3}(\mathbb{C})\setminus
	\{\mathbf{0}\})/ \mathbb{C} ^*$ and $\mathbf{s}$ a lift to $M_{3 \times
		3}(\mathbb{C})\setminus \{\mathbf{0}\}$ of $S$. The matrix $\mathbf{s}$
	induces a non-zero linear transformation $s:\mathbb
	{C}^{3}\rightarrow \mathbb {C}^{3}$, which is not necessarily
	invertible. Let $Ker(s) \subsetneq \mathbb{C} ^3$ be its kernel and let
	$Ker(S)$ denote its projectivization to $\mathbb{P} _{\mathbb{C}} ^2$,
	taking into account that $Ker(S):= \varnothing$ whenever
	$Ker(s)=\{(0,0,0)\}$. We refer to
	\cite{CS0}
	for more details about
	this subject.
	
	\subsection{The limit set}
	There are two  types of limit sets relevant for this work. These are  
	the Kulkarni  limit set $\Lambda_{\Kul}$
	and the  Myrberg   limit set $\Lambda_{\Myr}$. Let us define these.
	Let $\Gamma\subset   \textrm{PSL}(n+1,\mathbb{C})$ be a discrete subgroup.
	
	\begin{definition}\label{dek Kul limit}\label{kuldef}
		The  \textit{Kulkarni limit set}
		of $\Gamma$ is:
		$$\Lambda (\Gamma) = L_0(\Gamma) \cup L_1(\Gamma) \cup
		L_2(\Gamma),$$  
		where:
		\begin{itemize}
			\item $L_0(\Gamma)$ is  the closure  of  the points in
			$\mathbb{P}^2_{\mathbb{C}}$ with infinite isotropy group.
			\item  
			$L_1(\Gamma)$ is the closure of the set of accumulation points of the
			$\Gamma $-orbits of  points in
			$\mathbb{P}^2_{\mathbb{C}}\setminus L_0(\Gamma)$.
			\item    
			$L_2(\Gamma)$ is  the closure of the union of accumulation
			points of $\{\gamma (K) : \gamma \in \Gamma\}$, where $K$ runs over
			all the compact sets in $\mathbb{P}^2_{\mathbb{C}}\setminus
			(L_0(\Gamma) \cup L_1(\Gamma))$.
		\end{itemize}
	\end{definition}
	
	This is  a   closed  $\Gamma$-invariant set. Its complement
	$$\Omega_{\Kul} (\Gamma) = \mathbb{P}^n_{\mathbb{C}} \setminus
	\Lambda(\Gamma),$$
	is the \textit{Kulkarni discontinuity region}
	of $\Gamma$. We know from
	\cite {Kulkarni}
	that the $\Gamma$-action on  $\Omega_{\Kul} (\Gamma) $ is properly discontinuous, and we further know from
	\cite{BCN1}
	that $\Omega_{\Kul} (\Gamma)$ contains the equicontinuity region ${\rm Eq}(\Gamma)$.
	However,  as we will see later, $\Omega_{\Kul} (\Gamma)$ generically is the largest open subset of $\P^2$
	where the group acts properly and discontinuously, but this is not always so.
	
	Recall that  ${\rm Eq}(\Gamma)$ is the set of points in
	$\P^2$   for which there is an open
	neighborhood where the family of transformations defined by $\Gamma$ is
	a normal family.
	
	Now, given $\Gamma\subset   \textrm{PSL}(n+1,\mathbb{C})$, a discrete subgroup, we  let  $\Gamma'$ be the set of pseudo-projective maps of $\P^2$ which are limits of sequences in $\Gamma$. That is:
	$$\Gamma' :=\large  \{ S \in  \SP(3,\C) \,|\,  S \; \hbox{is an accumulation point of} \; \Gamma \large \} \;.
	$$
	The following notion is due to Myrberg
	\cite{Myr}
	(cf.   \cite[Definition 3.3]{BCN1}):
	\begin{definition}\label{def Myr}
		The {\it Myrberg limit set} of $\Gamma$ is:
		$${\Lambda_{\Myr}}(\Gamma) := {\cup_{S \in \Gamma'} Ker (S)} \;,
		$$
		where $Ker (S)$ is the kernel of the pseudoprojective transformation $S$ defined above.
	\end{definition}
	
	Notice that if a pseudoprojective map $S$ is not in $\PSL(3,\C)$, then its kernel is either a line or a point. The following useful notion is introduced in \cite{BCN4}:
	
	\begin{definition}\label{def effective}
		We say that a line $\ell \subset \P^2$  is an effective line for  $\Gamma$
		if there exists a pseudoprojective transformation $S \in  \Gamma'$ with kernel $\ell$.
	\end{definition}
	The Myrberg limit set of $\Gamma$ contains all the effective lines, and it is not hard to see that, by
	\cite{Na},
	it contains the Kulkarni limit set.
	It is immediate  from \cite{Myr}
	that $\Gamma$  acts properly and discontinuously on the complement of ${\Lambda_{\Myr}}(\Gamma)$. Furthermore,
	one has the following theorem.
	
	\begin{theorem}
		The Myrberg limit set is  the complement of the equicontinuity region:
		$\;\Eq(\Gamma) = \P^2 \setminus {\Lambda_{\Myr}}(\Gamma) \,.$ Also, the  set $ {\Lambda_{\Myr}}(\Gamma) $ equals the union of all effective lines of $\Gamma$ except when
		it is the disjoint union of one line and one point. Moreover, $\Lambda_{\Kul}(\Gamma)\subset  {\Lambda_{\Myr}}(\Gamma)$.
	\end{theorem}
	
	The first statement  is a special case of \cite[Lemma 4.1]{CS0}
	(see also \cite[Theorem 3.4]{BCN4}).
	The second statement is Corollary 4.5 in
	\cite{BCN4}.

	Throughout this work, when we say limit set we refer to the Kulkarni limit set, otherwise we will say it explicitly.
	
	\subsection{Classification of the elements in $\PSL(3,\C)$}\label {ss classification of elements}\label {subsec. elements}
	Recall that the elements of $\PSL(2,\C)$ are  classified as
	elliptic, parabolic or loxodromic: $g$ is elliptic if regarded as a M\"obius transformation in $S^2 \cong \P^1$, up to  conjugation it is a rotation; parabolic elements are  translations up to conjugation  and  loxodromic elements are multiplication by a complex number of norm $\ne 1$. Equivalently,
	given $g \in \PSL(2,\C)$ and a lift $\tilde g$ to $\SL(2,\C)$, $g$ is elliptic if $\tilde g$ is diagonalizable with unitary eigenvalues, it is parabolic if $\tilde g$ is not diagonalizable and it has unitary eigenvalues, and $g$ is loxodromic otherwise. The classification given in these terms extends to all dimensions (see \cite{Na, CNS}).

	The \emph{elliptic} elements in $\PSL(3,\C)$ are those elements
	$\gamma$ that have a lift to $\SL(3, \C)$ whose Jordan canonical form
	is
	\begin{displaymath}
		\left(
		\begin{array}{ccc}
			e^{i\theta _1} & 0 & 0 \\
			0 & e^{i \theta _2} & 0 \\
			0 & 0 & e^{i \theta _3}
		\end{array}
		\right).
	\end{displaymath}
	The limit set  for (the cyclic group generated by) $\gamma$ elliptic is either empty
	or all of $\Proy ^2$, according to whether the order of
	$\gamma$ is finite or infinite. The subgroups of $\PSL(3, \C)$
	containing an elliptic element of infinite order cannot be discrete.
	
	The \emph{parabolic} elements in $\PSL(3, \C)$ are the elements
	$\gamma$ with limit set $\Lambda_{\Kul}(\gamma)$  equal to one
	single complex line. If $\gamma$ is parabolic then it has a lift to
	$\SL(3,\C)$ with Jordan canonical form one of the following
	matrices:
	\begin{displaymath}
		\left(
		\begin{array}{ccc}
			1 & 1 & 0 \\
			0 & 1 & 0 \\
			0 & 0 & 1
		\end{array}
		\right) , \left(
		\begin{array}{ccc}
			1 & 1 & 0 \\
			0 & 1 & 1 \\
			0 & 0 & 1
		\end{array}
		\right) , \left(
		\begin{array}{ccc}
			e^{2\pi i t} & 1 & 0 \\
			0 & e^{2\pi i t} & 0 \\
			0 & 0 & e^{-4\pi i t}
		\end{array}
		\right) \,, \, e^{2 \pi i t} \ne 1 \;.
	\end{displaymath}
	In the first case $\Lambda_{\Kul}(\gamma)$ is the complex line consisting
	of all the fixed points of $\gamma$, in the second case
	$\Lambda_{\Kul}(\gamma)$ is the unique $\gamma$-invariant complex line. In
	the last case $\Lambda_{\Kul}(\gamma)$ is the complex line determined by
	the two fixed points of $\gamma$.
	
	There are four kinds of \emph{loxodromic} elements in $\PSL(3,\C)$:
	\begin{itemize}
		
		\item The \emph{complex homotheties} are the elements  that
		have a lift to $\SL(3,\C)$ with Jordan canonical form:
		\begin{displaymath}
			\left(
			\begin{array}{ccc}
				\lambda  & 0 & 0 \\
				0 & \lambda  &0 \\
				0 & 0 & \lambda ^{-2}
			\end{array}
			\right) , \quad |\lambda| \ne 1.
		\end{displaymath}
		The limit set $\Lambda_{\Kul}(\gamma)$ is the set of fixed points of
		$\gamma$, consisting of a complex line and a point.
		
		\item The \emph{screws} are those elements $\gamma \in \PSL(3, \C)$
		that have a lift to $\SL(3, \C)$ whose Jordan canonical form is
		\begin{displaymath}
			\left(
			\begin{array}{ccc}
				\lambda  & 0 & 0 \\
				0 & \mu  &0 \\
				0 & 0 & (\lambda \mu) ^{-1}
			\end{array}
			\right) , \quad \lambda \ne \mu, \, |\lambda| =|\mu| \ne 1.
		\end{displaymath}
		The limit set  consists of a complex line
		$l$, on which $\gamma$ acts as an elliptic transformation of $\PSL(2,
		\C)$, and the fixed point of $\gamma$ not lying in $l$.
		
		\item The \emph{loxoparabolic} elements $\gamma \in \PSL(3, \C)$ have
		a lift to $\SL(3, \C)$ whose Jordan canonical form is
		\begin{displaymath}
			\left(
			\begin{array}{ccc}
				\lambda  & 1 & 0 \\
				0 & \lambda  &0 \\
				0 & 0 & \lambda ^{-2}
			\end{array}
			\right) , \quad |\lambda| \ne 1.
		\end{displaymath}
		The limit set $\Lambda_{\Kul}(\gamma)$ consists of two
		$\gamma$-invariant complex lines. The element $\gamma$ acts on one
		of these  complex lines as a parabolic element of $\PSL(2, \C)$ and
		on the other as a loxodromic element of $\PSL(2, \C)$.
		
		\item The \emph{strongly loxodromic} elements $\gamma \in \PSL(3, \C)$
		have a lift to $\SL(3, \C)$ whose Jordan canonical form is
		\begin{displaymath}
			\left(
			\begin{array}{ccc}
				\lambda _1  & 0 & 0 \\
				0 & \lambda _2 &0 \\
				0 & 0 & \lambda _3
			\end{array}
			\right) , \quad |\lambda_1| < |\lambda_2|< |\lambda_3|.
		\end{displaymath}
		This kind of transformation has three fixed points, one of them is
		attracting, another is repelling and the other point  is a saddle. The
		limit set  is  the union of the complex
		line determined by the attracting and saddle points and the complex
		line determined by the saddle and repelling points.
	\end{itemize}

	\subsection{Elementary groups in $\PSL(3,\C)$}
	In the previous section we have examples of elements in $\PSL(3,\C)$ where the limit set consists of one line, or one line and one point, or two lines. We have the following theorem  from \cite{BCN4}.
	
	\begin{theorem}\label{th. number of lines}
		Let $\Gamma \subset \PSL(3,\mathbb{C})$ be an infinite discrete subgroup. Then:
		\begin{enumerate}
			\item The number of complex projective lines in  $\Lambda_{\Kul}(\Gamma)$
			is equal to $1, 2, 3$ or $\infty$.
			\item  The number of lines $\Lambda_{\Kul}(\Gamma)$ in general position can be $1, 2, 3, 4$ or $\infty$.
			\item If the number of lines in $\Lambda_{\Kul}(\Gamma)$ is exactly 3, then these lines are in general position.
			
			\item If there are infinitely
			many  lines in $\Lambda_{\Kul}(\Gamma)$, then  the effective lines form a  perfect set in $\Lambda_{\Kul}(\Gamma)$.
			\item There can be at most one isolated point in  $\Lambda_{\Kul}(\Gamma)$, and in that case this limit  set is the disjoint union  of that point and one line.
			
		\end{enumerate}
	\end{theorem}
	
	Statements (i), (ii) and (iv) are theorems 1.1 and 1.2  in  \cite{BCN4}.
	Statement (iii) is a corollary of \cite[Proposition 5.4]{BCN4}.  
	Statement (v) is not in the literature so we now give a short proof of it.

	Recall that a discrete subgroup of $\PSL(2,\C)$ is elementary if its limit set has finite cardinality.
	We remark that in the case of $\PSL(3,\C)$ there are  groups with infinitely many lines in their limit set $\Lambda_{\Kul}$, but only two of them in general position. In view of the previous theorem, this leads to the following definition:
	
	\begin{definition}
		A discrete subgroup of $\PSL(3,\C)$ is elementary of the first kind if its limit set $\Lambda_{\Kul}$ has finitely many lines. The group is elementary of the second kind if $\Lambda_{\Kul}$ has finitely many lines in general position.
	\end{definition}
	
	So every elementary group of the first kind also is of the second kind, but not conversely. In the following sections we describe the classification of elementary groups.

	\subsection{The control group}\label{s:control}
	We refer to \cite{CNS}
	for a discussion about this section. Consider  $\Gamma\subset \PSL(3,  \mathbb{C})$  a (discrete or not)
	subgroup which acts on $\P^2$ with a point $p$ which is  fixed by all of $\Gamma$. Choose an
	arbitrary line $\ell$ in $\P^2 \setminus \{p\}$, and notice we have
	a canonical projection:
	$$\pi=\pi_{p,\ell}:\mathbb{P}^2_{\mathbb{C}}\setminus\{p\}\longrightarrow
	\ell\,,$$   given by $\pi(x)=\overleftrightarrow{x,p}\cap \ell$. It
	is clear that this map is holomorphic and it allows us to define a
	group homomorphism:
	$$\Pi=\Pi_{p,\ell}:\Gamma\longrightarrow Bihol(\ell) \cong \PSL(2,\C) \,,$$
	by $\Pi(g)(x)=\pi(g(x))$.
	If we choose another line, say $\ell'$, one gets similarly a projection
	$\pi'=\pi_{p,\ell'}:\mathbb{P}^2_{\mathbb{C}}\setminus\{p\}\to
	\ell'\,,$ and a group homomorphism $\Pi'=\Pi_{p,\ell'}:\Gamma \to
	\PSL(2,\C)$.
	It is an exercise to see that $\Pi$ and  $\Pi'$ are equivalent in the sense that there is a biholomorphism $h: \ell \to \ell'$ inducing an automorphism $H$ of $ \PSL(2,\C)$ such that  $H \circ \Pi = \Pi'$. As before, the line
	$\ell$ is called {\it the horizon}.
	
	\medskip
	This leads to the following definition:
	
	\begin{definition} \index{Kleinian group ! weakly semi-controllable}
		Let $\Gamma\subset \PSL(3,\C)$ be a discrete group as above. We  call $\Pi=\Pi_{p,\ell}$ the control morphism (or map) and its image
		$\Pi(\Gamma) \subset \PSL(2,\C)$, is the {\it control group}. These are
		well-defined and independent of $\ell$ up to an
		automorphism of $\PSL(2,\C)$.
	\end{definition}
	
	The control map and the control group allow us to get information about the dynamics of $\Gamma$ by looking at a subgroup of $\PSL(2,\C)$, which is far easier to handle. The prize we pay is that the control group in  $\PSL(2,\C)$ may not be discrete.
	
	\section{Purely parabolic groups}\label{sec parabolic}
	
	We now follow \cite {BCNS}
	and  look at
	the  discrete subgroups in $\PSL(3,\C)$ that, besides the identity, have only parabolic elements. These are called purely parabolic and
	there are five families of such groups; three of them split into various subfamilies according to their limit set (and their control group, see
	\cite {CNS}).
	All of these are elementary.

	The simplest purely parabolic groups are  cyclic, generated by a parabolic element. As described above, there are three types of such elements in $\PSL( 3, \mathbb C)$,  described by the Jordan normal form
	of their lifts to $\SL(3, \mathbb C)$. Each of these belongs to a different type of the families we describe below. The first type  generates torus groups (see definitions below), the second generates Abelian Kodaira groups and the  ellipto-parabolic elements generate elliptic groups.

	\begin{enumerate}
		\item Elliptic groups. These are the only purely parabolic groups that are not conjugate to subgroups of the Heisenberg group ${\rm Heis} (3,\C)$ and they are   subgroups of  fundamental groups of elliptic surfaces. These have limit set a single line. Up to conjugation these groups are of the form:
		\[
		{{\rm Ell}}(W,\mu)=
		\left \{
		\left[
		\begin{array}{lll}
			\mu(w) & \mu(w)w&0\\
			0& \mu(w)&0\\
			0&0& \mu(w)^{-2}\\
		\end{array}
		\right ]
		:w\in W
		\right \},
		\]
		where $W\subset\mathbb{C} $ is an additive discrete subgroup  and $ \mu:W\rightarrow \mathbb{S}^1$ is a  group morphism.

		\item Torus groups. These are subgroups of fundamental groups of complex tori. These are of the form:
		\[
		\mathcal{T}(\mathfrak L)=
		\left
		\{
		\left[
		\begin{array}{lll}
			1 & 0 & a\\
			0 & 1 & b\\
			0 & 0 & 1\\
		\end{array}
		\right]:(a,b)\in \mathfrak L
		\right \} \;, \]
		where $ \mathfrak L$ is an additive discrete subgroup of $\C^2$. These groups also have a single line as limit set, so they are elementary of the first kind.
		\item Dual torus groups,
		\[
		\mathcal{T}^*(\mathfrak L)=
		\left
		\{
		\left[
		\begin{array}{lll}
			1 & a & b\\
			0 & 1 & 0\\
			0 & 0 & 1\\
		\end{array}
		\right]:(a,b)\in \mathfrak L
		\right \} \;. \]
		These split into three types: the first   have  Kulkarni limit set a complex projective line, so these are elementary of the first kind.
		The second type have  limit set a cone of projective lines over a circle, so these are of the second kind.
		The third type have all  $\P^2$ as limit set and they are non-elementary.

		\item Inoue groups and their extensions.  Inoue groups are proper subgroups of fundamental groups of Inoue surfaces. To define these, let $\mathfrak L \subset \mathbb{C}^2 $ be an additive discrete subgroup and consider a dual torus group.
		$$ {\mathcal I} = {\mathcal I}(u,v) :=  \left  \langle \begin{bmatrix}
			1 &u  &v\\
			0 & 1 & 0 \\
			0 & 0 & 1 \\
		\end{bmatrix}  \; , \; (u,v)\in \mathfrak L \right \rangle \;.
		$$
		Inoue groups are obtained by taking a generator $ {\mathcal I}$ as above and adding to it a generator of the form
		
		$$ \gamma_1= \gamma_1(x,y,z) :=
		\begin{bmatrix}
			1 & x+z & y\\
			0 & 1& z\\
			0 & 0& 1
		\end{bmatrix}\;  x ,y, z  \in \C \;.
		$$
		The
		limit set  is a cone of lines over a circle, so these are elementary of the second kind.
		
		Then one has the extended Inoue group, which are purely parabolic as well, with limit set all of $\P^2$, so they are non-elementary and do not fall within the scope of this article. We refer to \cite {BCNS}
		for details.

		\item Finally one has the Kodaira groups ${\rm Kod}_0$, which are Abelian,  and their extensions. A Kodaira group  is a discrete group in $\PSL(3,\C)$ such that each  element $h$ in the group can be written as:
		$$
		\left[
		\begin{array}{lll}
			1 & a&b\\
			0 & 1 & a\\
			0 & 0& 1\\
		\end{array}
		\right] \,.
		$$
		One can show that these are extensions of dual torus groups. Their limit set is a line, so they are elementary of the first kind.
		
		There are five types of extensions $\widetilde{\mathcal{K}_i}$, $i = 1, \cdots , 5$,
		which are
		purely parabolic and discrete. These are all  non-Abelian and they split into five types according to their limit set and the control group. The first type $\widetilde{\mathcal{K}_1}$ have limit set a projective line, so they are elementary of the first kind; the second type $\widetilde{\mathcal{K}_2}$ have limit set a cone of projective lines over the circle, so they are of the second kind, while the remaining three types have limit set all of $\P^2$ and they are non-elementary. We refer to \cite{BCNS}
		for details.
	\end{enumerate}
	

	\section {Solvable groups}
	An essential step towards understanding elementary groups in $\PSL(3,\C)$ is studying the dynamics of solvable groups in $\PSL(3,\C)$. Unlike the classical case of $\PSL(2,\C)$, now there is  great richness.
	
	\begin{definition}
		A group $H$ is called virtually solvable if   it contains a finite index subgroups $G$  and   subgroups $G_i$, $i= 0,...,k$, such that:
		\begin{itemize}
			\item $e = G_0 \subset  G_1 \subset \ldots \subset G_k =G$, and
			\item  $G_{j-1}$ 
			is normal in
			$G_j$, and 
			$G_j/G_{j-1}$
			is an abelian group, for $j = 1,\ldots, k$.
		\end{itemize}
	\end{definition}

	For instance, for a subgroup of ${\rm PSL}(2,\Bbb{C})$, virtually solvable is equivalent to saying that the group is elementary, {\it i.e.} its limits set has finite cardinality,  see\cite{CS} and \cite{Ma},   or in an equivalent way   (via Tit's dichotomy) a Möbius subgroup is non-virtually solvable if it contains either  a  non-Ciclyc  Schottky group or a purely elliptic free group whose rank is at least two,  see \cite{TTA, TI}  for a general  discussion  or \cite{Leon, Ma} for standard arguments in the one-dimensional case.  
	
	One has the following theorem from 
	\cite{Mau1}:

	\begin{theorem} For subgroups of ${\rm PSL}(3,\C)$, 
		solvable  is equivalent to having a finite index subgroup which is conjugate to group of upper triangulable matrices.
	\end{theorem}
	
	So the following are examples of solvable groups:
	
	\vskip.2cm
	\begin{enumerate}
		\item   All purely parabolic groups. 
		\item Hyperbolic Toral groups (we discuss these later, in section \ref {s:fourlines}).
		\item Fundamental groups of Inoue surfaces, see \cite{BCN2, BCNS, CS}. 
	\end{enumerate} 
	
	\vskip.2cm
	So, after our previous section \ref {sec parabolic}, 
	to have the complete picture of solvable groups we must look at solvable groups with loxodromic elements. This is done in
	\cite{Mau1, Mau2}. 
	In those articles the author determined the corresponding limit set and the representations of solvable groups in ${\rm PSL}(3,\Bbb{C})$ containing loxodromic element. The main results can be summarized as follows. These summarize 
	Theorems 1.1 to 1.4 in 
	\cite{Mau2}  
	and theorems 1 and 2 in \cite{Mau1}:

	\begin{theorem} 
		The  solvable groups in ${\rm PSL}(3,\Bbb{C})$ can be of three main types:
		\begin{enumerate}
			\item Commutative.
			\item Cyclic extensions of torus groups by a strongly loxodromic element.
			\item  Dual torus groups extensions.
		\end{enumerate}
	\end{theorem}

	\begin{theorem} 
		If a  solvable  group $\Gamma$ is commutative, 
		then its limit set consists of either two lines, three lines in general position, or it is a line and a point, and $\Gamma$ 
		can be either triangular or a fake Hopf group:
		\begin{itemize}
			\item Diagonal groups:
			\[
			\left \{
			diag(\alpha^n,\alpha^m,1):n,m\in \Bbb{Z}
			\right \}
			\]
			here $\alpha,\beta \in \Bbb{C} ^*$, $\vert \alpha \vert\neq 1 $ or $\vert \beta \vert\neq 1$.
			\item Fake Hopf:
			\[
			\left \{
			\begin{bmatrix}
				\mu(w) & w\mu (w)&0\\
				0 &\mu(w)& 0\\
				0 &  0 &\mu(w)^{-2}
			\end{bmatrix}
			\right \}
			\]
			where $W\subset \Bbb{C}$ is a discrete group and $\mu:W\rightarrow \Bbb{C}^* $ is a group morphism, satisfying some technical conditions, see \cite{Mau2}.
		\end{itemize}
	\end{theorem}

	\begin{theorem} If the group is an extension of   either a torus group or   a dual torus group, then its limit set can be   a single projective line, or  two lines, or else it has infinitely many lines but  either it is a cone of lines over a circle or it has 
		four lines in general position. And one has:
		
		\begin{enumerate}
			\item If the group $\Gamma$ is a cyclic extension of a torus group, then it is a semi-direct product of a torus group and the cyclic group generated by a strongly  loxodromic element.  
			
			\item If the group is an extension of a dual torus groups, then it  can be:
			\begin{itemize}
				\item An extension by a loxodromic element (of any type); or
				\item An extension by two loxo-parabolic elements, and in this case the group is a semi-direct product of the dual torus group and the group generated by the two loxodromic elements.
				
			\end{itemize}
		\end{enumerate} 
	\end{theorem}

	In the  sequel we describe all groups with limit set as stated in this theorem.

	\section{Groups with limit set exactly  one line}\label{s:oneline}
	In the previous section we described several purely parabolic groups with limit set  a single  line.  
	Here we follow  \cite{BCN3} 
	and describe all the  discrete 
	groups in $\PSL(3,\C)$ with limit set  a  line. In particular, every such group    is virtually
	nilpotent. 
	
	Let $\Gamma$ be complex Kleinian with limit set a line $\ell$. 
	Then $\Gamma$  acts
	properly and discontinuously on the complement of  $\ell$. From the above described classification of the elements in $\PSL(3,\C)$ we known that  
	the elements in $\Gamma$ are all either elliptic, parabolic or loxoparabolic. 
	If the group contains a loxoparabolic element, then one finds that $\Gamma$ is conjugate to a group such that every element can be represented by a matrix of the form:
	$$
	\left(
	\begin{array}{lll}
		a & 0&v\\
		0 & d & 0\\
		0 & 0& 1\\ 
	\end{array} 
	\right) \,,
	$$
	where  $\ell$ becomes the line $\overleftrightarrow{e_1,e_2}$. Moreover, $\Gamma$ acts as an Euclidean group on the line $\Flecha{e_1, e_3}$. By well-known facts on Euclidean groups (see \cite{Ma}), 
	$a$ is a root of unity of order $1, 2, 3, 4$ or $6$, and the rank of $\Gamma$ considered as acting on  $\Flecha{e_1,e_3}$ is equal to $1$ or $2$. 
	One can show that $a$ cannot be a root of unity of order $2,3,4$ nor $6$, otherwise $\Gamma$ would contain a complex homothety. Finally, the rank of $\Gamma$ considered as acting on 
	$\Flecha{e_1,e_3}$ is not equal to $1$, otherwise, the Kulkarni limit set of $\Gamma$ would be equal to two lines. 
	
	
	On the other hand, if the group does not contain any loxoparabolic elements, then there are two cases. If it acts on $\ell$ without parabolic elements then $\Gamma$ can be
	considered as a discrete group of Euclidean isometries of $\mathbb R^4$. If some element   acts on $\ell$ as parabolic element then the
	group can be identified with a group of triangular matrices. Then one can rule out 
	the existence of irrational ellipto-parabolic elements and show that  there exists a unipotent subgroup of finite index.
	
	In this way we arrive to the following  \cite[Theorem 1.1] {BCN3}: 
	
	\begin{theorem}
		Let $\Gamma$  be a subgroup of $\PSL(3,\C$) such that its Kulkarni limit 
		consists of precisely one complex projective line $\ell$. Then:
		\begin{enumerate}
			\item If $\Gamma$ does not contain  loxoparabolic elements nor
			an element which acts as a parabolic element on the line $\ell$, then $\Gamma$  is a group of isometries of $\C^2$ and it contains a free abelian normal subgroup of finite index and rank $\le 4$.
			
			\item If $\Gamma$ does not contain  loxoparabolic elements but it does contain 
			an element which acts as a parabolic element on $\ell$, then $\Gamma$  does
			not contain any irrational ellipto-parabolic elements and it is a finite extension
			of a unipotent subgroup that consists of unipotent parabolic maps.
			Hence it is a finite extension of a group of the form $\Z, 
			\Z^2, \Z^3, \Z^4$, $\Delta_k$ or $G_k$, where 
			$$\Delta_k = \large \langle A, B, C, D : \, C, D \; \hbox{are central and} \; \,[A,B] = C^k \large \rangle \,, k \in \mathbb N \;,
			$$
			$$ G_k  = \large \langle A, B, C : \, C \; \hbox{is central and} \; \,[A,B] = C^k \large \rangle \,, k \in \mathbb N \;,
			$$
			\item If $\Gamma$  does contain a loxoparabolic element, then $\Gamma$  is isomorphic to 
			the group $\Z\oplus \Z \oplus \Z_{n_0}$, where $n_0 \in \mathbb N$
			is arbitrary. The $\Z_{n_0}$ summand is a group of complex
			reflections, while $\Z\oplus \Z$ is generated by a loxoparabolic element and another element
			which can be loxoparabolic or parabolic. 
		\end{enumerate}
	\end{theorem}
	

	\section{Groups with limit set two lines}\label{sec 4} 
	
	We must distinguish two cases:
	
	a) Groups that have exactly two lines in their limit set. These are elementary of the first kind; and 
	
	b) Groups with more than two lines in the limit set, but only two  in general position.  In this case, we know from Theorem \ref {th. number of lines} that the total number of lines actually is infinite. These are elementary of the second kind.
	
	For instance, cyclic groups generated by a strongly loxodromic element have exactly two lines in their limit set. On the other hand, for instance, the dual torus groups of the second type, described before, have limit set a cone of projective lines over a circle. In this case, there are infinitely many lines in the limit set, but only two in general position.
	
	Let us describe these  elementary groups.
	As in 
	\cite{BCN4}, 
	for every open $\Gamma$-invariant set $U$
	we denote by $\lambda(U)$   the maximum number of complex
	projective lines  contained in $\mathbb{P}^2_{\mathbb{C}}\setminus
	U$, and by 
	$\mu(U)$  the maximum number of such lines in general position. If $U$ is the Kulkarni discontinuity region $\Omega_{\Kul}(\Gamma)$, then its complement $\mathbb{P}^2_{\mathbb{C}}\setminus
	U$ is the limit set  $\Lambda_{\Kul}= \Lambda_{\Kul}(\Gamma)$.

	\begin{theorem} \label{t:2lines}
		Let $\Gamma\subset {\rm PSL}(3,\mathbb{C})$ be a group satisfying $\mu(\Omega_{\Kul}(\Gamma))=2$, then either  the limit set  is the union of two lines  or it is a pencil  of lines over a nowhere dense  perfect set. Moreover, one of the following facts occurs:
		\begin{enumerate}
			\item if $\lambda(\Omega_{\Kul}(\Gamma))=2$, then $\Gamma$ is virtually solvable and contains  a loxodromic element, see  \cite{Mau1,Mau2} 
			for a precise description.    
			\item     if $\lambda(\Omega_{\Kul}(\Gamma))=\infty$, then $\Gamma$ either is  virtually purely parabolic  or contains a loxodromic element, and:
			\begin{enumerate}
				\item If the  group is virtually  purely parabolic then it contains a finite index subgroup conjugate to one of the following groups: a dual torus
				group of type II, a Kleinian Inoue group, or an extended Kodaira group
				$\mathcal{K}_2$. 
				\item If the group contains a loxodromic element, then either  $\Gamma$ is virtually conjugate to a solvable group
				(described in \cite{Mau2}),    
				or  it is  a  weakly controllable group  whose   control group is  non-elementary. 
			\end{enumerate}
		\end{enumerate}
	\end{theorem}

	\begin{proof} 
		Since  $\mu(\Omega_{\Kul}(\Gamma))=2$, we have that $\Gamma$ has a global fixed point $p$, the point where those two lines intersect, 
		and  $\lambda(\Omega_{\Kul}(\Gamma))$ is either 2 or $\infty$ by  Theorem 1.1 
		\cite{BCN4}.
		In fact, 
		will see in the next section that if $\lambda(\Omega_{\Kul}(\Gamma)) = 3$, then also $\mu(\Omega_{\Kul}(\Gamma)) = 3$. And we know from Theorem 
		\ref{th. number of lines} 
		that if $ \lambda(\Omega_{\Kul}(\Gamma)) > 3$, then this number actually is $\infty$.

		Also, if $\lambda(\Omega_{\Kul}(\Gamma))=2$, then,  we claim, that $\Gamma$ is virtually solvable. To see this notice that the group is 
		weakly semi-controllable. Choose a line $\ell$ as  the ''horizon". Then $\Lambda_{\Kul}$ meets $\ell$ in two points which are invariant under the control group $\Pi(\Gamma)$. Hence $\Pi(\Gamma)$ is virtually diagonalizable and therefore $\Gamma$ itself  is virtually diagonalizable. Hence it is virtually solvable. Then, 
		by Theorem  0.1 in \cite{BCNS},  
		we deduce that $\Gamma$  contains a loxodromic element. Then, by \cite{Mau1,Mau2}, 
		its limit set consists of exactly two lines. 
		
		In case  $\lambda(\Omega_{\Kul}(\Gamma))=\infty$,  we  have that  $\Gamma$ is either virtually purely parabolic  or  it contains a loxodromic element. In the first case by  part (b) item (2) of  Theorem 0.1 in \cite{BCNS} 
		we have  that $\Gamma$ contains a finite index subgroup conjugate to one of the following groups: a dual torus
		group of type II, a Kleinian Inoue group, or an extended Kodaira group $\mathcal{K}_2$, thus, in this case, the limit set is  a cone of lines over a Euclidean circle, see section 2 of \cite{BCNS}. 
		Now     the proof splits into the following cases:
		
		{\bf  Case 1.} The group $\Gamma$ is virtually solvable. In this case by Theorems  1.1, 1.2, 1.3, 1.4  in \cite{Mau2}, 
		the limit set of $\Gamma$  is either two lines or  a cone  of lines over an Euclidean circle.
		
		{\bf  Case 2.} The group $\Gamma$ is non-virtually solvable. In this case it is clear that the control group   $\Pi_{p,\ell}\Gamma$ is non-elementary, here $\ell$ is any line not containing $p$, and 
		\[
		\Lambda_{\Kul}(\Gamma)=
		\bigcup_{z\in \Lambda(\Pi(\Gamma)) }
		\overleftrightarrow{p,z}.
		\]
		
		That $\Pi_{p,\ell}\Gamma$ is non-elementary follows from the fact that, essentially by Borel's fixed point theory, a subgroup of $\PSL(2,\C)$ is elementary if and only if it is virtually solvable (see 
		\cite{CS}).
	\end{proof}

	
	\section{Groups with limit set three lines}\label{sec 5}
	We now describe the subgroups of $\PSL(3,\C)$ with  three lines in their limit set. 
	As in the previous section,  we denote by $\lambda = \lambda(\Omega_{\Kul}(\Gamma))$   the maximum number of complex
	projective lines  contained in $\Lambda_{\Kul}$, and by 
	$\mu = \mu(\Omega_{\Kul}(\Gamma))$  the maximum number of such lines in general position. Then
	there are groups with $\mu = \lambda = 3$, and there are groups  with $\lambda = \infty$ but $\mu =3$.  In the case $\lambda = \infty$ the groups in question are ``suspensions'' of non-elementary Kleinian groups in $\PSL(2,\C)$, see  Chapter 5  in \cite{CNS}:

	\begin{theorem}
		
		Let $\Gamma\subset {\rm PSL}(3,\mathbb{C})$ be a group for which $\mu=3$. Then $\lambda$ can be either $3$ or $\infty$ and: 
		\begin{enumerate}
			\item If $\lambda = 3$, then these lines are in general position  and $\Gamma$ contains a finite index subgroup   conjugate to: 
			\[
			G_0=
			\left \{
			diag[
			\alpha^n,\beta^m, 1]: n,m\in \mathbb{Z} 
			\right \} \;,
			\]
			where $\alpha,\beta \in \mathbb{C}^*$ are non-unitary complex numbers.
			\item If $\lambda = \infty$, then the limit set is a cone of lines plus a non-concurrent line and  $\Gamma$ contains a finite index subgroup   conjugate to: 
			\[
			H_{\Sigma,\rho, \alpha}=
			\left \{
			\begin{bmatrix}
				\alpha^n\rho([g])& 0 \\
				0 &g \\
			\end{bmatrix}: n\in \mathbb{Z}, g\in {\rm SL}(2,\mathbb{C}),\,\,  [g]\in \Sigma
			\right \} \;,
			\]
			where  $\alpha\in \mathbb{C}^* $ is a non-unitary  complex number, $\Sigma$  is a non-elementary discrete  group in ${\rm PSL}(2,\mathbb{C})$ with  non-empty discontinuity region, and $\rho$ is a function  $\Sigma\rightarrow \C^*$. 
			In this case 
			\[
			\Lambda_{\Kul}(H_{\Sigma,\rho, \alpha})
			=\overleftrightarrow{[e_2],[e_3]}\cup \bigcup_{p\in \Lambda(\Sigma)} \overleftrightarrow{[e_1],p} \;.
			\]
			
		\end{enumerate}
		
	\end{theorem}

	\begin{proof} 
		We   know from statement (iii) in Theorem \ref{th. number of lines} that if $\lambda = 3$ then  $\mu = 3$.
		That is, if the group has exactly three lines in its limit set, then these lines are in general position. 
		
		Now we assume $\mu = 3$ and 
		we choose three distinct  lines in general position contained in $\Lambda_{\Kul}(\Gamma)$, say $\ell_1,\ell_2,\ell_3$. We define 
		$$P=\{p_{ij}=\ell_i\cap \ell_j:  1\leq i<j\leq 3\}\,,$$   then any line contained in  $\Lambda_{\Kul}(\Gamma)$ must intersect $P$.

		We will say $p\in P$ is a vertex whenever there are infinite lines in $\Lambda_{\Kul}(\Gamma)$ passing through it. 
		If the set of vertices is empty, then by  Proposition 5.6 in \cite{BCN4}
		the set  $P$ is a $\Gamma$-invariant set.  Then $H=Isot(P, \Gamma)$ is a finite index subgroup  of $\Gamma$. Moreover, $\Lambda_{\Kul}(H)=\Lambda_{\Kul}(\Gamma)$ and $H$  is conjugate  to a group $G_0$ where every element has a diagonal lift. So by  Theorem 1.1 in \cite{Mau2} 
		we conclude 
		$
		G_0=
		\{
		diag[\alpha^n, \beta^m,1]:n,m\in \mathbb{Z}
		\}
		$ for some $\alpha,\beta \in \mathbb{C}^*$ non-unitary complex numbers. 
		
		If the set of vertices  is non-empty then by  the proof of Proposition 6.5 in \cite{BCN4} 
		we must have that this set is $\Gamma$-invariant. Observe that in case the number of vertices is  at least two,  we must have four lines in general position (compare with Proposition 6.5 in \cite{BCN4}) 
		which is not possible, therefore we have exactly one vertex, say $v$,  which is  $\Gamma$-invariant. Let $\ell\in \{\ell_1,\ell_2,\ell_3\}$ be a line not containing $v$. Since $\mu(\Omega_{\Kul}(\Gamma))=3$, we conclude that $\ell$ is $\Gamma$-invariant. After conjugating with a projective transformation, if necessary, let us assume that $v=[e_1]$ and $\ell=\overleftrightarrow{[e_2],[e_3]}$, as in Section \ref{s:control} take $\pi=\pi_{v,\ell}$ and $\Pi=\Pi_{v,\ell}$, now we claim:

		{\bf Claim 1.} If  $\Pi( \Gamma)$ is non-virtually solvable, then $\Pi(\Gamma)$ is discrete. On the contrary let us assume that $\Pi(\Gamma)$ is non-discrete, thus by Theorem 1 and Proposition 12 in \cite{Leon} 
		the principal component  of $\overline{\Pi(\Gamma)}$ is  ${\rm PSL}(2,\mathbb{C})$ or conjugate to   ${\rm SO}(3)$ or ${\rm  PSL}(2,\mathbb{R})$.   In the first two cases,  the orbit  of any line passing through $v$ is dense in $\ell$, and this is not possible. In the latter case, let $(g_n)\subset \Pi(G)$ be a sequence such that $\lim_{n\to \infty }g_n=Id$ and $h_n=[g_1,g_n]\neq Id$ for $n\geq 2$. If  $G_m\in \Gamma$ satisfies $\Pi(G_m)=g_m$, then a straightforward computation shows  $\lim_{n\to \infty }[G_1,G_n]=Id$, which is a contradiction.

		{\bf Claim 2.} The group  $\Gamma$ cannot be  virtually  solvable. On the contrary, let us assume that $\Gamma$ is  solvable, then by Theorem 0.1 in \cite{BCNS} 
		the group $\Gamma$ contains a loxodromic element. By Theorems 1.1, 1.2,  1.3 and 1.4  in \cite{Mau2}, 
		there is $ W \subset  \mathbb{C}$
		a discrete additive subgroup  with $rank(W) \leq 2$  and $\alpha,\beta\in \C^* $ such that $\alpha\neq 1$ and $\alpha \beta^2\notin \mathbb{R}$ and  $\Gamma$ is  conjugate to: 
		
		\[
		\left \{
		\begin{bmatrix}
			1 & 0 & 0\\
			0 & 1 & w\\
			0 & 0  &1 
		\end{bmatrix}
		: w\in W
		\right  \}
		\rtimes
		\left  \{
		\begin{bmatrix}
			\alpha^{2n}\beta^n & 0 & 0\\
			0 & \alpha^{n}\beta^{2n}  & 0\\
			0 & 0  &1 
		\end{bmatrix}
		:n\in \mathbb{Z}
		\right \}.
		\]
		Now a simple computation shows $W$ is non-discrete, which is a contradiction.
		
		By the previous claims, we conclude that $\Pi(\Gamma)$ is a  discrete non-elementary with  a non-empty ordinary region. Now the rest of the proof is trivial.
	\end{proof}

	\section{Groups with limit set four lines in general position lines}\label{s:fourlines}
	In this section we provide an algebraic characterization of the subgroups
	of $\PSL(3,\C)$ with exactly four lines lines in general
	position  in their limit set $\Lambda_{\Kul}$ . Of course that by Theorem  \ref{th. number of lines} this implies that $\Lambda_{\Kul}$  actually has infinitely many lines, but only four in general position. The basic reference  is \cite{BCN2}.

	We recall first that an element $A \in \SL(2, \Z)$ is a hyperbolic toral automorphism if none of
	its eigenvalues has norm 1 (see \cite{Ka-Ha}).

	\begin{definition}
		A subgroup $\Gamma$ of $\PSL(3,\C)$ is  a {\it hyperbolic
			toral group} if it is conjugate to a  group of the form:
		$$
		\Gamma_A \, =  \left\{\, \left(
		\begin{array}{lll}
			A^k & b\\
			0 & 1 \\
		\end{array} 
		\right) \,   \; \Big| \; b=  \left(
		\begin{array}{lll}
			b_1 \\
			b_2  \\
		\end{array} 
		\right) \, ,\; b_1, b_2, k \in \Z\, \right\}
		$$
		where  $A \in \SL(2, \Z)$ is a hyperbolic toral automorphism.
	\end{definition}
	
	
	One has:
	
	\begin{theorem}\label{4lines} 
		Let $\Gamma  \subset  \PSL(3,\C)$ be a discrete group. The maximum number of
		complex lines in general position contained in its Kulkarni limit set is equal to four
		if, and only if, $\Gamma$  contains a hyperbolic toral group  whose index is at
		most 8. [In this case the limit set of the subgroup coincides with that of $\Gamma$, and the discontinuity region has four components each one diffeomorphic to $\Hip \times \Hip$, where
		$\Hip$ denotes the half open plane in $\C$.]
	\end{theorem}
	
	To prove the theorem above, the first step is showing that
	a toral group $\Gamma_A$ as above is the subgroup of  $\PSL(3,\Z)$ generated by the unipotent parabolic elements $P_1, P_2$ and the strongly loxodromic element $L$ given by:
	$$ P_1 =  \left(
	\begin{array}{lll}
		1 & 0& 0\\
		0 & 1  & 1 \\
		0 & 0 & 1 \\
	\end{array} 
	\right) \; , \; 
	P_2 =  \left(
	\begin{array}{lll}
		1 & 0& 1\\
		0 & 1  & 0 \\
		0 & 0 & 1 \\
	\end{array} 
	\right) \; , \; 
	L =  \left(
	\begin{array}{lll}
		A & 0\\
		0 & 1   \\
	\end{array} 
	\right) \; . \; 
	$$
	Therefore $\Gamma_A$ is discrete.
	The dynamical properties of the hyperbolic toral group are used  to show that 
	$\Lambda _{\textrm{\Kul}}(L)$ is contained in Kulkarni limit set of $\Gamma _A$.
	Also, a straightforward computation shows that $\Lambda _{\textrm{\Kul}}(P_1)$
	is contained in $\Lambda _{\textrm{\Kul}}(\Gamma _A)$. Thus, 
	$$\Lambda _{\textrm{\Kul}}(\Gamma _A)= \overline{\bigcup _{\gamma \in \Gamma _A} \Lambda _{\textrm{\Kul}}(\gamma)},$$
	because both sets in this equation contain three lines in general position
	(see Theorem 1.2 in \cite{BCN1}).
	
	Now, it is possible to compute the sets $\Lambda _{{\Kul}}(\gamma)$ for every
	$\gamma \in \Gamma _{A}$. One finds that $\Lambda _{\textrm{Kul}}(\Gamma _A)$ can be described as 
	a union of two pencils of lines over two distinct circles,  with vertices two distinct points, which share the line determined by the two vertices. Therefore, the maximum number of lines in general
	position contained in $\Lambda _{\textrm{Kul}}(\Gamma _A)$ is equal to four.
	Moreover, this description shows that the Kulkarni discontinuity region of $\Gamma _A$ is a disjoint
	union of four copies of $\Hip \times \Hip$, 
	where $\Hip$ denotes the upper half plane in $\C$. 
	
	Conversely, if we assume that the maximum number of lines in general position contained in 
	the Kulkarni limit set of a group $\Gamma$ is equal to four, then one shows that there are two distinguished points
	(called vertices) contained in infinitely many lines in this limit set.
	The subgroup which stabilizes these two vertices is denoted by $\Gamma _0$ and it has index at most two in $\Gamma$. Furthermore, it can be proved that $\Omega _{\textrm{Kul}}(\Gamma _0)$ has four components, each one diffeomorphic to a copy of $\Hip \times \Hip$. Finally, the subgroup of 
	$\Gamma _0$ stabilizing each of these components has finite index (at most four) in $\Gamma _0$, and it can be shown that it is a hyperbolic toral group.
	
	
	A description of the quotient spaces $\Omega _{\textrm{Kul}}(\Gamma _A)/ \Gamma _A$
	is given by the following theorem (see \cite{BGN}.)
	
	\begin{theorem}\label{solteorema}
		If $\Gamma _A$ is a hyperbolic toral group then
		\begin{enumerate}
			\item The group $\Gamma _A$ is isomorphic to a lattice of the group Sol (see Definition \ref{sol}). 
			\item The quotient space $\Omega _{\textrm{Kul}}(\Gamma _A)/ \Gamma _A$ is a disjoint union of four 
			copies of $$\left(\textrm{Sol}/\Gamma _A \right)\times \mathbb{R}.$$
		\end{enumerate}
	\end{theorem}
	
	\begin{definition}\label{sol}
		The group Sol is $\mathbb{R}^3$ equipped with the operation
		$$\left( 
		\left( 
		\begin{array}{c}
			x_1 \\
			y_1
		\end{array}
		\right), t_1 \right)
		\left( 
		\left( 
		\begin{array}{c}
			x_2 \\
			y_2
		\end{array}
		\right), t_2 \right)
		=
		\left( 
		\left( 
		\begin{array}{c}
			x_1 +e^{t_1}x_2\\
			y_1+e^{-t_1}y_2
		\end{array}
		\right), t_1+t_2 \right)$$
	\end{definition}
	
	The proof of \ref{solteorema} is done by giving an explicit smooth foliation of 
	$\Hip \times \Hip$ where the leaves are diffeomorphic to Sol, and by showing that 
	$\Hip \times \Hip$ is diffeomorphic to $\textrm{Sol} \times \mathbb{R}$ by a 
	$\Gamma _A$-equivariant diffeomorphism. 
	
	It can be shown that $\Gamma_ A$ is isomorphic to $\mathbb{Z}^2 \rtimes _A \mathbb{Z}$, and this group is determined by the conjugacy class in $\textrm{GL}(2, \mathbb{Z})$ of $A$ (see \cite{DeLaHarpe}
	pages 96,97). It follows from Theorem \ref{4lines} that there is a countable number of
	nonisomorphic Complex Kleinian groups such that the maximum number of lines in general position contained in its Kulkarni limit set is equal to four.

	
	
	\section{Groups with limit set having an isolated point}\label{s:isolated}
	
	Let us suppose now that $\Gamma\subset {\rm PSL}(3,\mathbb{C})$ is a group with at least an isolated point in its limit set. Then the results explained in this article show that $\Gamma$ must have only one line in its limit set, for otherwise $\Lambda_{\Kul}$ is a union of lines.
	
	Let $\ell$ be the unique line in $\Lambda_{\Kul}(\Gamma)$ and let $p\in \Lambda_{\Kul}(\Gamma)$ be an isolated point. Then $\ell$ is an invariant set and we look at the action of $\Gamma$ on $\ell$ to deduce from it consequences on the global dynamics of $\Gamma$.

	We note that this  is somehow similar to the controllable groups, where we have a global fix point and a line $L$ that we call the horizon, which is not invariant but we get an action on it. Then we know a lot about the action of the group on $\P^2$ by looking at the induced action on $L$. In the case we now envisage, we have the invariant line $\ell$ and the action on it says a lot about the dynamics of the group on $\P^2$. 
	There are two possibilities as  
	the action of $\Gamma$ on $\ell$ can be either:
	
	\begin{enumerate}	
		\item[a.] Virtually Solvable.
		
		\item [b.] Non- Virtually solvable.

	\end{enumerate}

	The idea of 
	the proof consists in showing first that the non-virtually solvable case cannot happen, so  the action must be virtually solvable. This uses the Tits Alternative. 
	
	In the solvable case, the dynamics implies that on the invariant line $\ell$ we must have either one or two fixed points. Then one shows that in the first case, when there is only one fixed point, one actually has an invariant flag and up to conjugation  the elements are all of the form
	
	\[
	\begin{bmatrix}
		\alpha &  0 & 0  \\
		0 & \beta & \delta  \\
		0 & 0 & \gamma 
	\end{bmatrix} \;.
	\]
	
	In the second case we find that, up to conjugation, all the elements are diagonal, and the group consists of lodrodromic elements that are screws (see Section \ref {subsec. elements}), so the limit set of each element consists of a point and a line. Then we show that the limit set of the whole group is  a point and a line.

	The above conisderations lead to the following:
	
	\begin{theorem}
		Let $\Gamma\subset {\rm PSL}(3,\mathbb{C})$ be a group with at least an isolated point in its limit set. Then the limit set  consists of exactly one line and one point, and  $\Gamma$ is  conjugate to either:  
		\begin{enumerate}
			\item 
			\[
			G_0=
			\left \{
			{\rm diag}\,[
			\alpha^n, e^{2\pi i\theta}, e^{-2\pi i \theta}]: n,m\in \mathbb{Z} 
			\right \}
			\]
			where $\theta \in \mathbb{R}$ and  $\alpha \in \mathbb{C}^*$  is  non-unitary complex number.
			\item or 
			\[
			H_{0}=
			\left \{
			\begin{bmatrix}
				\mu(w)^{-2}  & 0 &0\\
				0&	\mu(w)& w\mu(w)   \\
				0 &0 & \mu(w) \\
				
			\end{bmatrix}: w\in W\right \}
			\]
			where $W \subset \mathbb{C}$ is a discrete additive group and $\mu:(\mathbb{C}, +) \rightarrow (\mathbb{C}^*, \cdot)$ 
			is a group morphism satisfying that for every sequence of
			distinct elements $(w_n) \subset  W$, the limit $\lim_{n\to \infty }\mu(w_n)$ is either $0$ or $\infty$.
			
		\end{enumerate}
		
	\end{theorem}

	\begin{proof}  
		By Theorem 1.2 in \cite{BCN2} 
		we have $\mu(\Omega_{\Kul}(\Gamma))\leq 2$, and from Theorem \ref{t:2lines} we conclude $\mu(\Omega_{\Kul}(\Gamma))=1$.   On the other hand,  by Theorem 0.1 in \cite{BCNS}, 
		the group $\Gamma$ contains a Loxodromic element. Let $\ell$ be the unique line in $\Lambda_{\Kul}(\Gamma)$, $p\in \Lambda_{\Kul}(\Gamma)$ be an isolated point and  $\gamma\in\Gamma$ a loxodromic element. We must consider the following cases:

		{\bf Case 1.} The  action of $\Gamma$ restricted to $\ell$ is given by a  virtually solvable  group. In this case the group $\Gamma$ is virtually  solvable and 
		the theorem follows  from Claim 2 and Theorems 1,1, 1.2, 1.3 and 1.4 in \cite{Mau2}, 
		see also section  \ref{sec 1} above .  
		
		{\bf Case 2.} The  action of $\Gamma$ restricted to $\ell$ is  given by a  non- virtually solvable  group.  Under this assumption we claim:
		
		{\bf Claim}.  There are  $g_1,g_2\in \Gamma$ such that  the group generated by $\{\gamma_1\vert_\ell , \gamma_2\vert_\ell\}$  is a rank two free group which is either purely loxodromic or purely elliptic. If the action of $\Gamma$ on $\ell$ is given by a discrete group $H$, then  it is well know that  $H$ contains  a rank two Shottky group, see Chapter X in 
		\cite{Ma}, 
		these will do the job.  
		If the action on $\ell$ is given by a non discrete group $G$,   thus by Theorem 1 and Proposition 12 in 
		\cite{Leon} 
		the principal component  of $\overline{G}$ is  ${\rm PSL}(2,\mathbb{C})$,  or conjugate to   ${\rm SO}(3)$ or ${\rm  PSL}(2,\mathbb{R})$. In any such case  we deduce that $G$ contains a non-virtualy  solvable finitelly generated subgroup,  as a consequence of Tit's alternative  we deduce that   $G$   contains  a rank two free subgroup which is either purely loxodromic or purely elliptic. Which proves the claim.\\

		Now let  $g_1,g_2$ be the elements given by the previous claim. After conjugating with a projective transformation we can assume that $\ell=\overleftrightarrow{[e_2],[e_3]}$ so $h=[g_1,g_2]$ is a loxodromic element and 
		$$
		h=
		\begin{bmatrix}
			1 & 0 &0\\
			c & d & e\\
			f & g & h
		\end{bmatrix}
		$$

		where $dh-eg=1$ and $(d+h)^2\in \mathbb{C}-\{4\}$. Thus $h$ is either   an  elliptic element of infinite order,  which cannot happen by the  discreteness of $\Gamma$, or    a strongly loxodromic  element whose attracting 
		and repelling points are contained in $\ell$. We conclude that the attracting and  repelling  lines of $h$ are distinct from $\ell$. On the other hand, from the $\lambda$-Lemma  (see Lemma 7.3.6 in \cite{CNS}) 
		we know that the attracting and  repelling  line of $h$ are both contained in $\Lambda_{\Kul}(\Gamma)$, which is a contradiction.
	\end{proof}
	

	\bibliographystyle{amsplain}

\end{document}